\documentclass[a4paper,11pt]{amsart}
\usepackage{amsmath, amsthm, amssymb}
\usepackage[percent]{overpic}
\usepackage{graphicx}
\usepackage{mathbbol}
\usepackage{txfonts}
\usepackage[usenames]{color}
\usepackage{srcltx} 
\newtheorem{theorem}{Theorem}
\newtheorem{thm}{Theorem}[section]
\newcommand{\bt}{\begin{thm}}
\newcommand{\et}{\end{thm}}

\newtheorem{cor}[thm]{Corollary}   
\newcommand{\bc}{\begin{cor}}
\newcommand{\ec}{\end{cor}}

\newtheorem{lem}[thm]{Lemma}   
\newcommand{\bl}{\begin{lem}}
\newcommand{\el}{\end{lem}}

\newtheorem{proposition}{Proposition}
\newtheorem{prop}[thm]{Proposition}
\newcommand{\bp}{\begin{prop}}
\newcommand{\ep}{\end{prop}}

\newtheorem{defn}[thm]{Definition}
\newcommand{\bd}{\begin{defn}}    
\newcommand{\ed}{\end{defn}}

\newtheorem{rmrk}[thm]{Remark}   

\newcommand{\br}{\begin{rmrk}}
\newcommand{\er}{\end{rmrk}}

\newcommand{\be}{\begin{equation}}

\newcommand{\ee}{\end{equation}}

\newcommand{\R}{\mathbb{R}}

\newcommand{\dist}{\operatorname{dist}}
\newcommand{\diam}{\operatorname{Diam}}

\newcommand{\Ricci}{\rm{Ric}}

\newcommand{\vol}{\operatorname{Vol}}

\newcommand{\norm}[1]{\left\lVert#1\right\rVert}

\newcommand{\ot}{\overline{\theta}}

\begin{document}


\title{Gromov-Lawson tunnels with estimates}


\author{J\'ozef Dodziuk}
\address{CUNY Graduate Center and Queens College}
\email{jdodziuk@gmail.com}





\begin{abstract}
In an appendix to an earlier paper \cite{BDS-sewing} we showed how to construct tunnels
of positive scalar curvature and of arbitrarily small length and volume connecting points in a \emph{three dimensional}
manifold of \emph{constant sectional curvature}. Here we generalize the construction to arbitrary dimensions and
require only positivity of the scalar curvature.

\end{abstract}
\maketitle

\section{Introduction}
Suppose $X$ is a Riemannian manifold of positive scalar curvature of dimension $n\geq 3$. Gromov and Lawson
\cite{Gromov-Lawson-tunnels2}, \cite{Gromov-Lawson-tunnels} and independently Schoen and Yau \cite{Schoen-Yau-tunnels}
proved that if a manifold $M$ obtained from $X$ by a surgery on a sphere of codimension greater than or equal to 3
then $M$ carries a metric of positive scalar curvature. Obstructions to the existence of such metrics had been known
previously cf. \cite{Lich62}, \cite{Lich63}, \cite{Hitch74}.
The breakthrough provided by Gromov-Lawson and Schoen-Yau constructions led to a great deal of understanding of which
smooth manifolds carry metrics of positive scalar curvature.

More recently Rosenberg and Stolz \cite{Rosenberg-Stolz-PSC2001} revisited the construction of Gromov and Lawson to correct a mistake in the proof. Our reason for doing this again is to obtain an estimate on the size of the part
of the manifold where modification takes place. Before making this statement precise we need to introduce some
terminology and notation. As Gromov and Lawson point out the case of connected sums is the most important since it generalizes easily to surgeries on spheres of positive dimensions. Thus we will consider only this case.

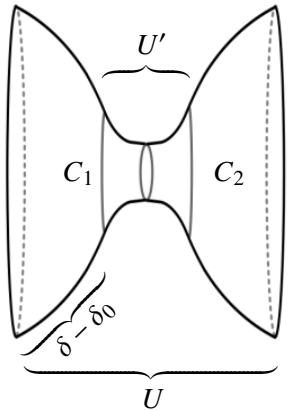
\begin{figure}[h]
\begin{overpic}[width=0.5\textwidth,trim=20 20 20 20, clip]{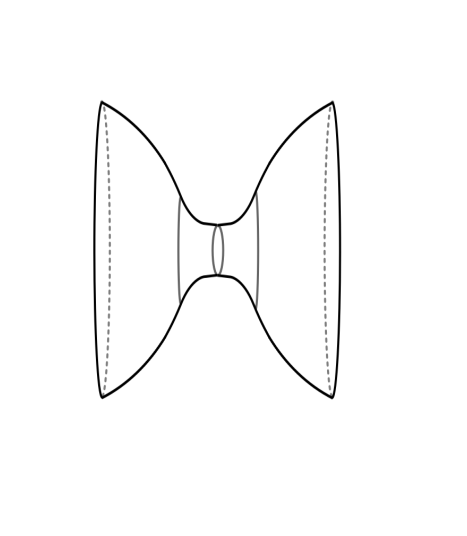}
\put(25,52){$C_1$}
\put(52,52){$C_2$}
\put(18,19){$\underbrace{\hphantom{\qquad\qquad\qquad\qquad\,\,\,\,}}$}
\put(39,11){$U$}
\put(32,70){$ \overbrace{\hphantom{\qquad\quad}} $}
\put(38,75){$U'$}
\put(17,21){\rotatebox{45}{$\underbrace{\hphantom{-----}}$}}
\put(23,20){\rotatebox{45}{$\delta - \delta_0$}}
\end{overpic}
\setlength{\abovecaptionskip}{-11pt}
\setlength{\belowcaptionskip}{-8pt}
\caption{The tunnel}
\label{Fig.SY-GL-tunnel}
\end{figure}

Suppose $p_1\neq p_2$ are two points of (possibly disconnected) $X$. Let $B(p_i,\delta)\subset X$, $i=1,2$, be disjoint balls in $X$ with the radii $\delta$ smaller than the injectivity radii at $p_1$ and $p_2$. The connected sum $M$ is obtained by removing the two balls from $X$ and glueing in a cylindrical region $U$ diffeormorphic to 
\mbox{$S^{n-1} \times [0,1]$}. Thus
$$
M = X \setminus \left( B(p_1,\delta) \cup B(p_2,\delta) \right) \sqcup U.
$$
We fix $\delta_0 \in (0, \delta)$ sufficiently small to be specified later. The collars $C_i = B(p_i,\delta) \setminus B(p_i,\delta_0)$ are identified with subsets of $M$ and the Riemannian metric $g$ to be constructed on $U$ will agree 
with the original metric $h$ of $X$ (see the Figure~\ref{Fig.SY-GL-tunnel}). We will call the set $U$ the tunnel and
prove the following theorem.

\begin{theorem}
There exists a constant $\delta_0\in (0,\delta)$ and a Riemannian metric $g$ on $U$ with positive scalar curvature such that
\begin{align}
{}& g\,|\,C_i = h\,|\,C_i\qquad \mbox{for} \quad i=1,2 \\
{}& g \quad\mbox{has positive scalar curvature}\\
\label{diam-vol}
{}& \diam{U} = O(\delta)\quad \mbox{and}\quad \vol{U} = O(\delta^n).
\end{align}

More precisely, the set $U' = U \setminus \left(C_1 \cup C_2\right)$ satisfies $\diam{U'} = O(\delta_0)$ and $\vol{U'}=O(\delta_0^{n})$.
\end{theorem}

Note that the earlier constructions in \cite{Gromov-Lawson-tunnels2}, \cite{Gromov-Lawson-tunnels}, \cite{Schoen-Yau-tunnels}, and \cite{Rosenberg-Stolz-PSC2001} did not give information about the size of the tunnels. In \cite{BDS-sewing},
we proved the theorem above for manifolds of three dimensions and constant positive sectional curvature. This was
sufficient for constructing examples of sequences of manifolds of positive scalar curvature whose limits (under any reasonable notions of convergence) did not have positive generalized scalar curvature in the sense explained in \cite{BDS-sewing}. 
The main result of this paper removes the restriction on dimension and allows variable sectional curvature.

\section{Outline of the proof}
In this section we establish the notation, describe the setup and outline the construction of tunnels.
Let $X$ be a Riemannian manifold of positive scalar curvature, $D\subset X$ a geodesic ball centered at
$p\in X$ of small radius $\delta$. Using the normal geodesic coordinates $x^1,x^2,\ldots x^n$, the metric
on $D$ is obtained by considering $D=\{ x^1e_1+x^2e_2+\ldots x^ne_n \mid \norm{x} \leq \delta \}$,
where $e_1,e_2\ldots e_n$ is an orthonormal basis of $T_pX$, and pulling back the metric of $X$ to D via the
exponential map. We set $r(x) = \norm{x} $ to be the distance of $x$ to the origin of D and figure
$S^{n-1}(\rho) = \{x\in D \mid r(x)=\rho\}$. We will construct a new metric on $D\setminus \{0\}$ of 
positive scalar curvature. Following Gromov and Lawson \cite[Section 1]{Gromov-Lawson-tunnels2}
we consider the Riemannian product $D\times \R$ and a suitable curve $\gamma$ (to be described and constructed
below). We then define a hypersurface $M\subset D\times \R$ as $M=\{(x,t)\in D\times \R \mid (r(x),t) \in \gamma\}$.
It is useful to think of $M$ of as a hypersurface of revolution around the $t$-axis of a curve in $(x^1,0,\ldots 0,t)$-plane. This is a correct interpretation only if the metric of $D$ has \emph{constant} sectional curvature but it is
a very good approximation of the true picture if the radius of $D$ is very small. The requirements on $\gamma$ are 
that it begins along the positive $r$-axis and ends as a horizontal line segment $r=r_\infty$. Thus the metric on $M$
extends the metric of $D$ near $\partial D$ and finishes as the product metric of the form $S^{n-1}(r_\infty)\times \R$. Of course,
the main requirement on $\gamma$ is that the resulting hypersurface $M$ has positive scalar curvature and that the 
length of $\gamma$ is $O(\delta)$. Construction of $\gamma$ is the main difficulty of the proof. Note that $S^{n-1}(r_\infty )$ is not a round sphere. However, Lemma 1 of \cite{Gromov-Lawson-tunnels2} (quoted below) allows us to modify the
metric of $M$ near the end of the tube so that the modified metric is a product of the round sphere of radius $r_\infty$ with
an interval. This will allow us to connect two such tubes to form a tunnel of very small length and volume and
of positive scalar curvature.

\begin{figure}[h]
\begin{overpic}[width=\textwidth,trim=0 190 0 0,clip]{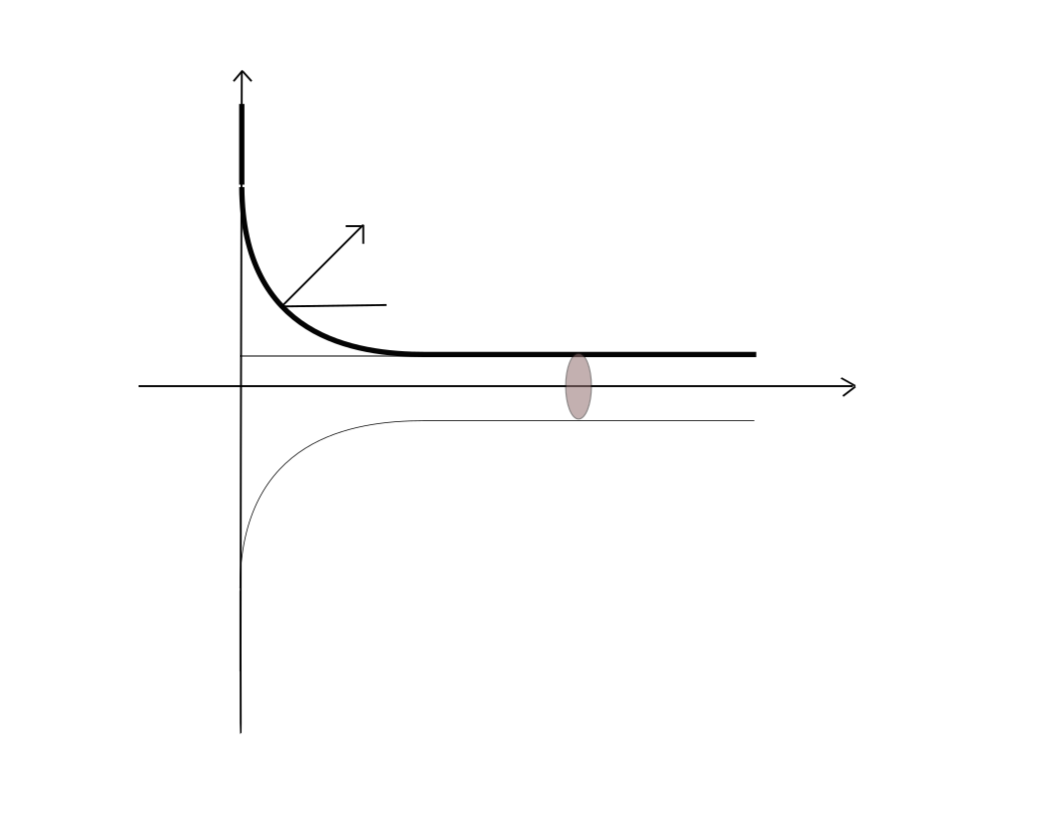}
\put(63,27){\Large $\gamma$}
\put(26,33){\Large $n$}
\put(31,31){\Large $\theta$}
\put(78,18){\Large $t$}
\put(18,50){\Large $r$}
\put(18,25){\Large $r_\infty$}
\put(18,40){\Large $\delta_0$}
\end{overpic}
\setlength{\abovecaptionskip}{-11pt}
\setlength{\belowcaptionskip}{-8pt}
\caption{The curve $\gamma$}
\setlength{\abovecaptionskip}{-11pt}
\setlength{\belowcaptionskip}{-8pt}
\label{curve-gamma}
\end{figure}
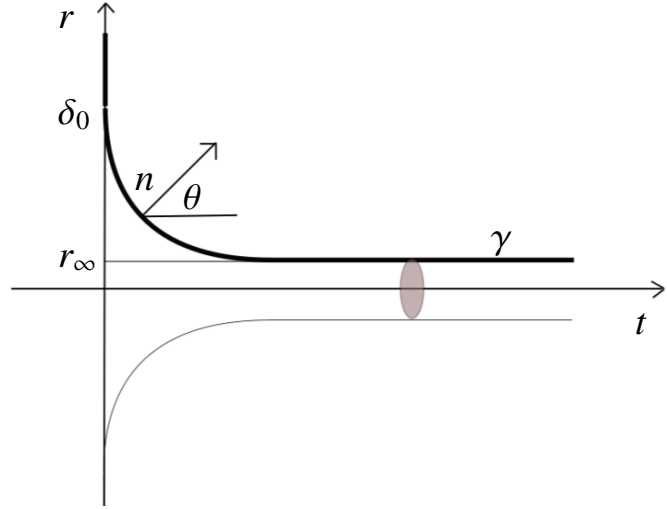

The following lemma of \cite{Gromov-Lawson-tunnels2} describes how the small geodesic spheres in $X$ differ from
round spheres of the same radius.
\bl \label{lemma1}
The principal curvatures of the hypersurfaces $S^{n-1}(\epsilon)$ in $D$ are each of the form $-{1}/{\epsilon}
+O(\epsilon)$ for small $\epsilon > 0$. Furthermore, let $g_\epsilon$ be the induced metric on $S^{n-1}(\epsilon)$ and
let $g_{0,\epsilon}$ be the standard round metric of curvature ${1}/{\epsilon^2}$. Then as $\epsilon \rightarrow 0$,
$({1}/{\epsilon^2})g_\epsilon \rightarrow (1/{\epsilon^2})g_{0,\epsilon}=g_{0,1}$ in the $C^2$ topology.
As a matter of fact, for an appropriate choice of the norm on the space of $C^2$ tensors, $
\norm{({1}/{\epsilon^2})g_\epsilon - g_{0,1}}=O(\epsilon^2)$.
\el

The estimate of principal curvatures  above and the Gauss curvature equations lead to the following expression
for the scalar curvature $\kappa$ at the point $(x,t) \in M$, cf.~formula (1) of \cite{Gromov-Lawson-tunnels2}.
\be \label{sc-curv}
\begin{split}
\kappa = \kappa^D & - 2\,{\Ricci}^D\left (\frac{\partial}{\partial r},\frac{\partial}{\partial r}\right )\sin^2\theta \\
                  & + (n-1)(n-2)\left(\frac{1}{r^2} + O(1)\right)\sin^2\theta \\
                  & - (n-1)\left( \frac{1}{r}+O(r) \right)k \sin \theta
\end{split}
\ee
where $\kappa^D(x,t)=\kappa^D(x)$ is the scalar curvature of $D$ at $x$, ${\Ricci}^D$ is the Ricci tensor of $D$ at
$x$, $k$ is the geodesic curvature of $\gamma$ at the point $(r(x),t)\in \gamma$, and $\theta$ is the angle between
the normal to $\gamma$ and the $t$-axis. We shall switch the order of variables $r$ and $t$ because we picture the $t$-axis as horizontal.

We write $s$ for the arc length parameter along $\gamma$ and define $\gamma$ by specifying its geodesic curvature
$k(s)$. Recall (cf.~Theorem 6.7, \cite{Gray-CurveSurfBook}) how $\gamma$ is determined by $k$. The unit tangent vector to $\gamma$ and the curvature are
given by

$$ \frac{d\gamma}{ds} = (\sin \theta , -\cos \theta )\qquad \mbox{and} \qquad k=\frac{d\theta}{ds}.$$

Thus, if $\gamma(s)$ is defined for $s\leq a$ and $k(s)$ is given for $s\geq a$, $\gamma(s)=(t(s),r(s))$ can
be extended as follows.
\be \label{def-eqn}
\begin{split}
\theta(s) &= \theta(a) + \int_a^s k(u)\,du\\
t(s) &= t(a) + \int_a^s \sin\theta(u)\,du\\
r(s) &= r(a) - \int_a^s \cos\theta(u)\,du.
\end{split}
\ee

We remark that a segment of $\gamma$ is a circular arc if and only if $k$ is constant. During our construction,
$\theta$ will increase from $0$ to $\pi/2$ so that the point $\gamma(s)$ will be moving down and to the right as
$s$ increases. We will construct $\gamma$ as a sequence of circular arcs, i.e. choosing $k(s)$ to be piecewise
constant. This will only be a $C^1$ curve that will be smoothed out in Section \ref{smooth}. However, $\theta(s)$, $r(s)$, and $t(s)$
will have no discontinuities as the formulae above show. We will also show that the scalar curvature of $M$ is positive
at all points where $\kappa$ is defined.

\section{The construction of a $C^1$ curve} 
Begin by choosing a point $(0,\delta_0)$ in the $(t,r)$-plane with $\delta_0$ small to be specified later. $\gamma$ runs down the $r$-axis for $s\leq 0$ with $\gamma(0)=(0,\delta_0)$. As the
initial segment of $\gamma$ for positive $s$ choose an arc of a circle of curvature $k=1$, tangent to the $r$-axis at $\delta_0$. We continue this arc for sufficiently small length
$s_0\leq \delta_0/2$ choosing $s_0$ to insure that that $\kappa$ stays positive if $s\leq s_0$. Such a choice is possible by
(\ref{sc-curv}) since $\theta(s)=0$ for $s\leq 0$ and $\kappa^D$ has a positive lower bound on $X$. Let $\theta(s_0)=\theta_0$ and $\gamma(s_0)=(t_0,r_0)$. Observe that $\theta_0$ is \emph{positive}.
The following lemma gives a sufficient condition for the scalar curvature $\kappa$ of $M$ to be positive.
\bl \label{sufficient}
For a sufficiently small $\delta_0$ and $s\geq s_0$, $\kappa$ will be positive provided
\be \label{condition}
\frac{\sin \theta}{4r} > k.
\ee
\el
\begin{proof}
We rewrite the right-hand side of (\ref{sc-curv}) as follows.
\begin{equation*}
\begin{split}  
&\frac{(n-1)(n-2)}{2r^2}\sin^2\theta +\left (\frac{(n-1)(n-2)}{2r^2} + O(1)\right )\sin^2\theta \\
-\,& 2(n-1)\frac{1}{r} k \sin \theta + \left ((n-1)\frac{1}{r}-O(r)\right ) k\sin\theta\\
+\,&\kappa^D
\end{split}
\end{equation*}

Note that the term containing the Ricci tensor in (\ref{sc-curv}) was absorbed in the term with $O(1)$ on the first line.
$\kappa^D > 0$ by assumption and so are the second terms on the first and second lines above provided that $r\leq \delta_0$ is sufficiently
small.
Thus $\kappa$ will be positive if
$$
\frac{(n-1)(n-2)}{2r^2} \sin^2\theta -\frac{2(n-1)}{r}k\sin\theta >0.
$$

It is here that we have to make the the choice of $\delta_0$ so that $r(s)$ is sufficiently small to make the terms that we dropped positive; $r(s)\leq r_0 \leq \delta_0$. Now we cancel 
common factors and use the fact that $n\geq 3$ to obtain our sufficient condition

$$
\frac{\sin\theta}{4r} -k >0.
$$
\end{proof}
Note that the inequality (\ref{condition}) is, except for the constant factor in the denominator, the same as
inequality (230) in \cite{BDS-sewing}. It is therefore not surprising that we will carry out the
construction in a very similar way to the construction in that paper. We now proceed to the main step in the construction. Suppose $\gamma(s)=(t(s),r(s))$
has been defined up to $s=a\geq s_0>0$. We extend it by a circular arc of constant curvature $k=\frac{\sin \theta(a)}{8r(a)}$ for
length $\Delta s=r(a)/2$. Since $\frac{\sin \theta(s)}{r(s)}$ is an increasing function $$
\frac{\sin \theta(s)}{4r(s)} > \frac{\sin \theta(a)}{4r(a)} > \frac{\sin \theta(a)}{8r(a)}=k$$
on $[a,a+\Delta s]$ so that the curvature condition (\ref{condition}) is satisfied. Moreover, since $\Delta s=r(a)/2$,
$\gamma$ will not cross the $t$-axis. Most importantly, by (\ref{def-eqn}), $\theta$ will increase by $$
\Delta \theta = k\Delta s= \frac{\sin \theta(a)}{8r(a)}\frac{ r(a)}{2} = \frac{\sin \theta(a)}{16} \geq 
\frac{\sin \theta(s_0)}{16},
$$
at least a fixed amount $\frac{\sin\theta_0}{16}$ \emph{independent of the starting point $\gamma(a)$.}

 We now proceed inductively beginning with $a=s_0$ by setting 
 \begin{equation} \label{induction}
 s_i=s_{i-1} + \Delta s_i , \qquad \Delta s_i=\frac{r_{i-1}}{2}, \qquad
 k_i=\frac{\sin \theta_{i-1}} {8r_{i-1}},
 \end{equation}
 where $r_j$ and $\theta_j$ denote $r(s_j)$ and $\theta(s_j)$ respectively. At every step, since $\theta(w)$ is increasing, the change in the angle
 $\Delta \theta_j =\theta_j -\theta_{j-1}$ is at least $\frac{\sin\theta_0}{16}$ so that
 \be \label{arb-large}
 \theta_i \geq \theta_0 + i\frac{\sin \theta(s_0)}{16}.
 \ee
 We are trying to construct a curve $\gamma$ for which the angle $\theta$ increases to $\pi/2$ when the tangent to 
 $\gamma$ becomes  horizontal. In any case, for the argument above
 we need $\sin \theta$ to be increasing which will be the case only if $\theta \leq \pi/2$. Since the right-hand side of
 (\ref{arb-large}) becomes arbitrarily large we can stop the construction if $\theta_{i-1}<\pi/2 \leq \theta_i$ at the value of $s$ determined by
 $\theta(s)=\pi/2$. This would produce a curve with a "full bend" but \emph{without an estimate of its length}.
 Up to now our argument is a variation of \cite[pp.\ 225--226]{Gromov-Lawson-tunnels}. To obtain a curve of controlled
 legth we break the induction off when the angle $\theta$ reaches the value $\overline{\theta}$ to be specified later
 but sufficiently close to
 $\pi/2$. We then complete the construction with a \emph{single}
 circular arc. Thus define $m$ so that $\theta_{m-1} < \ot \leq \theta_m$ and redefine $s_m$ so that 
 $\theta(s_m)=\theta_m=\ot$.
 
 The lemma below will allow us to estimate the length $s_m$.
 \bl \label{ratio}
 There exists a positive constant $C<1$ that depends only on $\ot$ such that for all $i$, $0\leq i\leq m$, $$
 \frac{r_{i}}{r_{i-1}} \leq C.$$
 \el
 \begin{proof}
 We compute $r_{i}$ explicitly using (\ref{def-eqn}) and (\ref{induction}).
 $$
 r_{i}= r_{i-1} - \int_{s_{i-1}}^{s_{i}} \cos \theta(u)\,du =r_{i-1} -\Delta s_{i} \cos \mu
 $$
for an angle $\mu \in [\theta_{i-1},\theta_{i}]$. Since $\Delta s_{i}=r_{i-1}/2$ for $i\leq m-1$, $\Delta s_m \leq r_{m-1}/2$ by
the definition of $s_m$, and $\theta_i \leq \ot$
$$
\frac{r_{i}}{r_{i-1}} = 1- \frac{1}{2}\cos \mu \leq 1-\frac{\cos\ot}{2}$$
i.e.~the lemma holds with $C=1-\frac{\cos\ot}{2}$.
\end{proof}
From the lemma above using (\ref{induction}) and the inequality $\Delta s_m \leq r_{m-1}/2$  $$
\frac{\Delta s_i}{\Delta s_{i-1}} \leq \frac{r_{i}}{r_{i-1}} \leq C.$$
Therefore \begin{equation}
\begin{split}
s_m= & s_0+\Delta s_1 +\ldots \Delta s_m \\
\leq & s_0 +\frac{r_0}{2} \left( 1 +C+\ldots +C^{m-1}\right)\\
\leq & s_0 +\frac{r_0}{2}\frac{1}{1-C}  \label{sm}
\end{split}
\end{equation}
which is $O(\delta_0)$.

So now $\gamma$ is defined on $[0,s_m]$ with $\theta(s_m)=\ot$. We show that we can achieve the "full bend"
to $\theta=\pi/2$ by extending $\gamma$ with a single circular arc. As above we need to define $k_{m+1} > 0$ and
$s_{m+1}=s_m+\Delta s_{m+1}$. We continue the added circular segment until $\theta (s_{m+1}) =\theta_{m+1} =
\pi/2$. By (\ref{def-eqn})
\begin{equation*}
\begin{split}
r_{m+1}= r(s_{m+1}) =& r_m - \int_{s_m}^{s_{m+1}} \cos \theta(u)\,du\\
= & r_m - \int_{s_m}^{s_{m+1}} \cos (s_m + k_{m+1}(u-s_m))\,du\\
= & r_m - \frac{1}{k_{m+1}} \left( \sin \theta_{m+1} -\sin \theta_m \right)\\
= & r_m - \frac{1}{k_{m+1}} \left( 1-\sin\ot \right).
\end{split}
\end{equation*}
So, for $r_{m+1}>0$ we must have  $$
k_{m+1}r_m > 1-\sin\ot .$$

On the other hand, the condition (\ref{sufficient}) will be satisfied if
$$k_{m+1}r_m < \frac{\sin\ot}{4}.$$
Thus we should choose $k_{m+1}$ so that
$$
 1-\sin\ot < k_{m+1}r_m < \frac{\sin\ot}{4}.$$
If $\sin\ot>4/5$ then $1-\sin\ot < \frac{\sin\ot}{4}$ and we can choose $k_{m+1}$ to satisfy both inequalities above. We fix $\ot \in (\sin^{-1}(4/5),\pi/2)$. Note that for such a choice
$$\Delta s_{m+1}=\frac{1}{k_{m+1}}\left( \frac{\pi}{2} - \ot\right)^{-1} < r_m (1-\sin\ot)^{-1}\left( \frac{\pi}{2} - \ot\right)$$
so that, for a fixed $\ot$,  $\Delta s_{m+1} = O(r_m)=O(\delta_0)$. Combining this with (\ref{sm}) we see that $s_{m+1}=O(\delta_0)$.

\section{Smoothing}\label{smooth}

Let us recapitulate the result of our construction so far. $k(s)$ is a piecewise constant function, $k|(s_i,s_{i+1}]=
k_{i+1}$ for $i=1,2,\ldots m$ and
$k_i < k_{i+1}$ for $i=1,2,\ldots m-1$.The resulting curve $\gamma$ is $C^1$ and piecewise $C^\infty$. It will have to be
smoothed out. Before smoothing extend $\gamma$ to the interval $[0,s_{m+1} + 2\delta_0]$ by a horizontal line segment, i.e.\ 
by setting $k=k_{m+2}=0$ on $[s_{m+1}, s_{m+1} + 2\delta_0]$. We set $S=s_{m+1} + 2\delta_0$ and smooth the
function $k(s)$ on $[0,S]$.

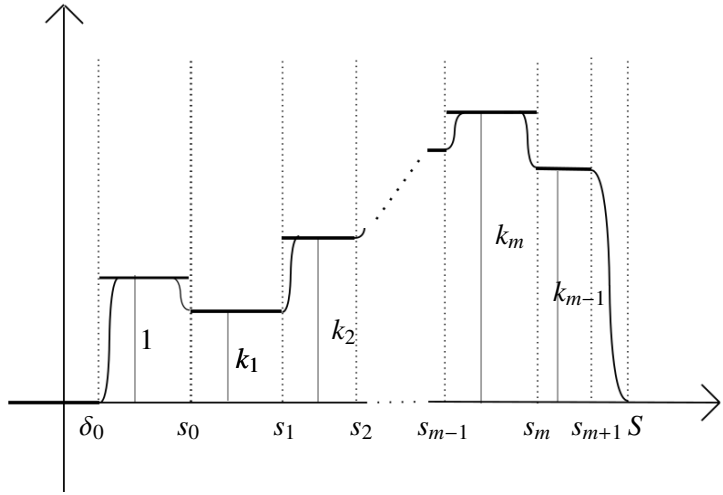
\begin{figure}[h]
\begin{overpic}[width=\textwidth,trim=0 200 0 0,clip]{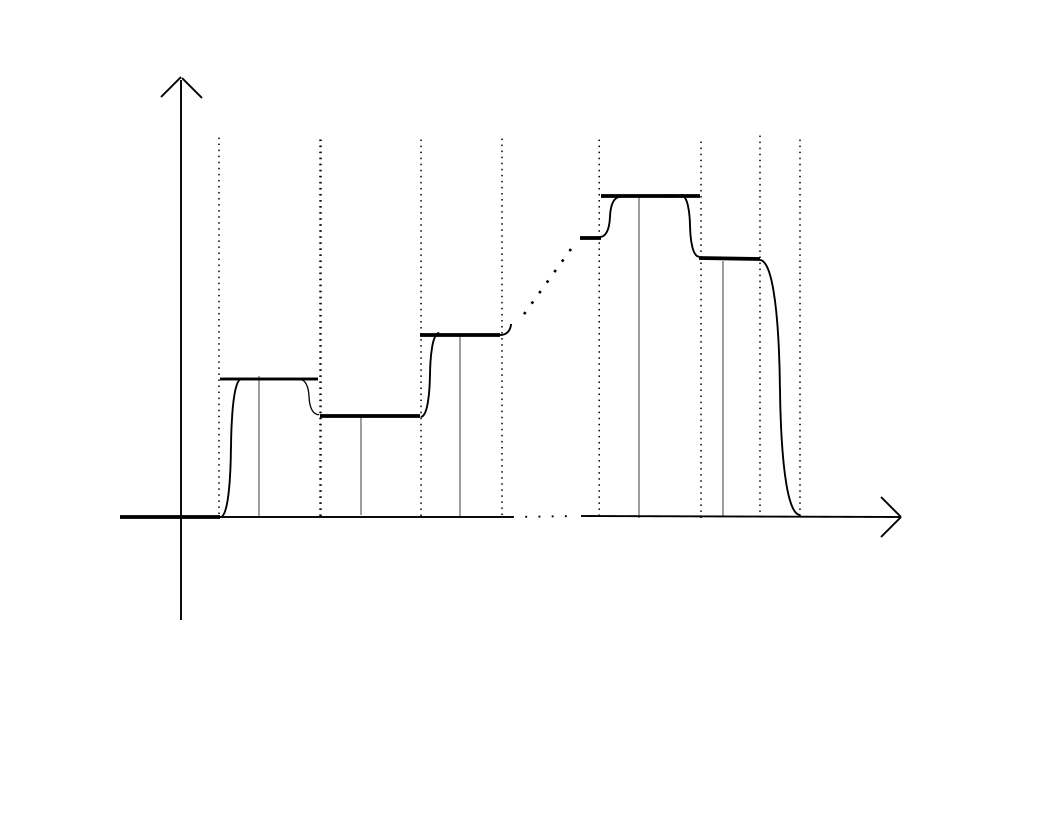}
\put(18,6){   $\delta_0$}
\put(28,6){   $s_0$}
\put(38,6){   $s_1$}
\put(46,6){   $s_2$}
\put(53,6){   $s_{m-1}$}
\put(69,6){   $s_{m+1}$}
\put(64,6){   $s_m$}
\put(75,6){   $S$}
\put(25,15){1}
\put(35,13){$k_1$}
\put(45,16){$k_2$}
\put(62,26){$k_m$}
\put(35,13){$k_1$}
\put(68,20){$k_{m-1}$}
\end{overpic}
\caption{Smoothing of the curvature function}
\setlength{\abovecaptionskip}{-11pt}
\setlength{\belowcaptionskip}{-8pt}
\label{smoothing}
\end{figure}

The smoothing procedure is elementary and is described in some detail in the appendix to \cite{BDS-sewing}. We recall
it briefly. Observe first that $k(s)$ can be approximated with arbitrarily high accuracy in $L^1$ by $C^\infty$ functions $\overline{k}(s)$ as, for example pictured in Figure~\ref{smoothing}. We can construct a family of functions
$k_\eta (s)=\overline{k}(s)$ for small positive $\eta$ so that on each of the intervals $(s_i,s_{i+1})$ of continuity $k(s)$ and  $\overline{k}(s)$ differ only near one or both end-points, $\overline{k}(s) \leq k(s)$ on $[0,s_{m+1}]$, and $\overline{k}(s)$ drops rapidly from $\overline{k}(s_{m+1})=k_{m+1}$ to zero in $[s_{m+1},S]$. We can achieve this together with
the requirement that $$
\int_0^S \overline{k}(s)\,ds = \frac{\pi}{2}.$$
Since $k_\eta (s)$ converges to $k(s)$ in $L^1$, the components $t_\eta(s)$, $r_\eta(s)$ and the normal angle $\theta_\eta(s)$ of
the curve $\gamma_\eta$ determined by $k_\eta$ will converge uniformly to $t(s)$, $r(s)$, and $\theta(s)$ respectively. This is
sufficient to conclude that for small $\eta$, the curve $\overline{\gamma}=\gamma_\eta$ will satisfy the condition (\ref{condition}) and will give rise to a hypersurface of positive scalar curvature. For details of the smoothing we refer to 
the appendix of \cite{BDS-sewing}.  We change the notation and from now on denote the smooth curve by $\gamma$. 

By choosing $\eta$ above sufficiently small we can assume that $\gamma$ is a horizontal line segment  
$r(s)=r(s_{m+1} +\delta_0)=r_\infty$ for $s\in [s_{m+1} + \delta_0, S]$ 
so that $M$ is isometric to $S^{n-1}(r_\infty) \times [s_{m+1} + \delta_0, S]$ near the end of the tube.

The construction of $\gamma$ and the resulting neck is now complete but our aim is to construct a tunnel by connecting two tubes. For that we will need to modify the metric of $M$ near its end. We first observe that the only place where the
ambient manifold enters the construction is the choice of $s_0$ on page \pageref{sufficient}. This in turn is dictated
by the bounds of the scalar curvature and the Ricci curvature on $D$. The same choice can be made for both disks that we are trying to connect. From then on the construction takes place entirely in the Euclidean $(t,r)$-plane and making the same choices
as we go along we obtain equal values of $r_\infty$ for both tubes. 

We now modify the metric near the end of the tube to make it round. Let $a=t(s_{m+1} +\delta_0)$, $b=t(S)$ and $\epsilon =r_\infty$. The induced metric on  the end of
the tube $\{ (x,t) \in M \mid a\leq t \leq b\}$ is $h_0=g_\epsilon + dt^2$ where $g_\epsilon$ is the induced metric on
$S^{n-1}(\epsilon)$. Recall that $\epsilon \leq \delta_0$. Let $h_1=\epsilon^2 g_{0,1} + dt^2$ where $g_\epsilon$ and
$g_{0,1}$ have the same meaning as in Lemma \ref{lemma1}. Let $\phi(t) = \psi((t-a)/\delta_0$ where $\psi(u)$ is a smooth function
on $[0,1]$ vanishing near zero, increasing to 1 at $u=3/4$ and equal to 1 for $u>3/4$. Define the new metric $h$ for
$t\in [a,b]$ as $$
h(x,t)= g_\epsilon(x,t) + \phi(t)\left( \epsilon^2 g_{0,1} - g_\epsilon \right) + dt^2.$$
For $t$ near $a$ $h=h_0$, the induced metric on $M$, while for $t$ near $b$, $h=h_1$, the round tube metric.
The first and second derivatives of $\phi$ are of order $O(\delta_0^{-1})$ and $O(\delta_0^{-2})$ respectively.
We consider $h_0$, $h_1$ and $h$ as tensors on the product of the standard sphere $S^{n-1}$ with the interval $[a,b]$ and see that $$
h(x,t) - h_0(x,t) = \phi(t)\left( \epsilon^2 g_{0,1}(x) - g_\epsilon (x)\right) =
\phi(t)\epsilon^2 \left(g_{0,1} - \frac{1}{\epsilon^2} g_\epsilon\right).
$$
It follows from Lemma \ref{lemma1} that the metric $h_0$ has positive (and very large) scalar curvature. Moreover,
since $\epsilon = r_\infty < \delta_0$,
 Lemma \ref{lemma1} and the bounds on derivatives of $\phi(t)$ show that all second order derivatives of
$h(x,t)-h_0(x,t)$ are of order $O(\delta_0^2)$. It follows that the scalar curvature of $h$ is positive provided
$\delta_0$ is chosen sufficiently small. Our construction of the tunnel is now complete.

Clearly, the metric constructed on $U$ has positive scalar curvature and agrees with the metric of $X$ near the
boundary. The estimates in (\ref{diam-vol}) follows since the length of the constructed tube is $O(\delta_0)$ and
each crossection $t=\mbox{const}$ is very close to a round sphere of radius $r(t)<\delta_0$. The proof of the
Theorem is now complete.

\section{Tunnels of prescribed length}

In view of possible applications we state the following consequence of the construction.

\begin{proposition}
Let $p_1\neq p_2 \in X$ be two points in a manifold $X$ of positive scalar curvature. Then for every $L>0$
there exists $\delta > 0$ and a Riemannian metric of positive scalar curvature on the tunnel $U$ as in Figure \ref{Fig.SY-GL-tunnel} such that 
\begin{itemize}
\item[(a)] The new metric agrees with the original metric on the collars $C_1$ and $C_2$.
\item[(b)] The distance $\dist {(C_1,C_2)} =L$ and the diameter $\diam {(U)} = O(L)$.
\item[(c)] $\vol (U) = O (L\delta^{n-1})$.
\item[(d)] For every continuous curve $\alpha$ connecting the two components of the boundary of $U$,
the tubular neighborhood of $\alpha$ of radius $2\pi\delta$ contains $U$ i.e.
$$\{x\in M \mid \dist (x,\alpha) \leq 2\pi\delta \} \supset U.$$
\end{itemize}
Moreover, $\delta$ can be chosen arbitrarily small.
\end{proposition}
\begin{proof}
We begin by choosing $\delta$ sufficiently small so that it is smaller than the injectivity radii at $p_1$ and
$p_2$ and apply the Theorem above to get a tunnel of diameter smaller than $L/2$. The central portion of the tunnel
is isometric to a round sphere of radius $r_\infty < \delta$ and an interval of legth smaller than $L/2$. We simply
replace this interval by one of appropriate length $l$ to make the distance between the collars equal to $L$  exactly, cf. Figure \ref{telescope}. Statements (c) and (d) follow since each crossection $t=\mbox{const}$ is very close to a round sphere of radius $r(t)<\delta$.

\begin{figure}[t]
\begin{overpic}[width=0.7\textwidth,trim=13 350 0 350,clip]{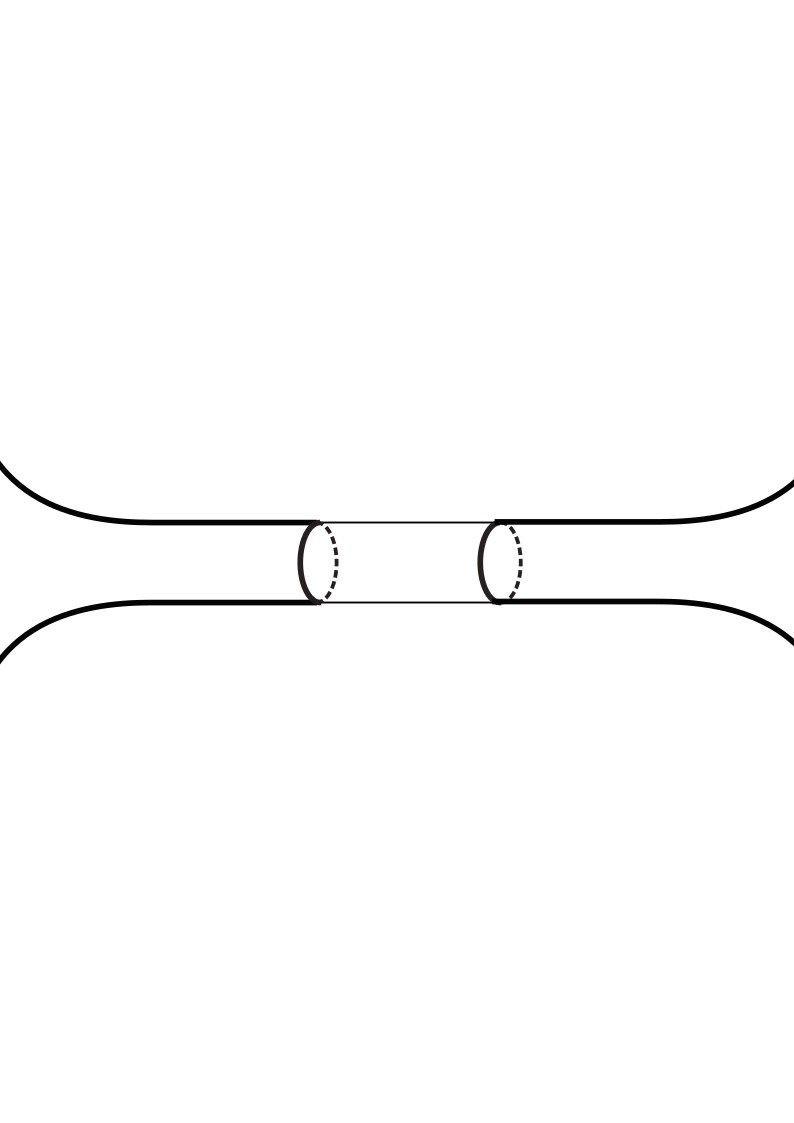}
\put(40,20){$\underbrace{\hphantom{------}}$}
\put(50,13){\large $l$}
\end{overpic}
\setlength{\abovecaptionskip}{-11pt}
\setlength{\belowcaptionskip}{-8pt}
\caption{Telescoping the tunnel}
\setlength{\abovecaptionskip}{-11pt}
\setlength{\belowcaptionskip}{-8pt}
\label{telescope}
\end{figure}

\end{proof}

\begin{thebibliography}{1}

\bibitem{BDS-sewing}
J.~Basilio, J.~Dodziuk, and C.~Sormani.
\newblock Sewing {R}iemannian manifolds with positive scalar curvature.
\newblock {\em J. Geom. Anal.}, 2017.

\bibitem{Gray-CurveSurfBook}
Alfred Gray.
\newblock {\em Modern differential geometry of curves and surfaces with
  Mathematica}.
\newblock CRC Press, second edition, 1998.

\bibitem{Gromov-Lawson-tunnels2}
Mikhael Gromov and H.~Blaine Lawson.
\newblock The classification of simply connected manifolds of positive scalar
  curvature.
\newblock {\em Ann. of Math.}, 111(3):423--434, 1980.

\bibitem{Gromov-Lawson-tunnels}
Mikhael Gromov and H.~Blaine Lawson, Jr.
\newblock Spin and scalar curvature in the presence of a fundamental group.
  {I}.
\newblock {\em Ann. of Math.}, 111(2):209--230, 1980.

\bibitem{Hitch74}
Nigel Hitchin.
\newblock Harmonic spinors.
\newblock {\em Advances in Math.}, 14:1--55, 1974.

\bibitem{Lich62}
Andr\'e Lichnerowicz.
\newblock Laplacien sur une vari\'et\'e riemannienne et spineurs.
\newblock {\em Atti Accad. Naz. Lincei Rend. Cl. Sci. Fis. Mat. Nat. (8)},
  33:187--191, 1962.

\bibitem{Lich63}
Andr\'e Lichnerowicz.
\newblock Spineurs harmoniques.
\newblock {\em C. R. Acad. Sci. Paris}, 257:7--9, 1963.

\bibitem{Rosenberg-Stolz-PSC2001}
Jonathan Rosenberg and Stephen Stolz.
\newblock {M}etrics of positive scalar curvature and connections with surgery.
\newblock In Andrew~Ranicki Sylvain~Cappell and Jonathan Rosenberg, editors,
  {\em {S}urveys on {S}urgery {T}heory}, number 149 in Annals of Mathematics
  Studies 2. Princeton University Press, 2001.

\bibitem{Schoen-Yau-tunnels}
R.~Schoen and S.~T. Yau.
\newblock On the structure of manifolds with positive scalar curvature.
\newblock {\em Manuscripta Math.}, 28(1-3):159--183, 1979.

\end{thebibliography}

\end{document}